\documentclass[12pt]{article}
\usepackage{amsmath,amssymb,amsthm}
\usepackage{tikz}

\usepackage{amsmath,graphicx}

\newcommand{\PV}{\mathcal{PV}}
\newcommand{\V}{\mathcal{V}}

\DeclareMathOperator{\wt}{\mathop{\rm wt}}
\DeclareMathOperator{\var}{\mathop{\rm var}}
\DeclareMathOperator{\BAR}{\mathop{\textsc{Bar}}}
\DeclareMathOperator{\Des}{\mathop{\rm{Des}}}
\DeclareMathOperator{\des}{\mathop{\rm{des}}}

\newtheorem{lemma}{Lemma}
\newtheorem{proposition}{Proposition}
\newtheorem{theorem}{Theorem}
\newtheorem{corollary}{Corollary}

\theoremstyle{definition}
\newtheorem{example}{Example}
\newtheorem{remark}{Remark}
\newtheorem{definition}{Definition}

\renewcommand{\bar}{\overline}

\newcommand{\abs}[1]{\lvert#1\rvert}

\title{Sign variation and descents}

\author{Nantel Bergeron\thanks{Supported by NSERC and the York Research Chair in applied algebra.} \qquad Aram Dermenjian\thanks{Support by NSERC} \qquad John Machacek\thanks{ supported by the York Science Fellowship, York University}\\
\small Department of Mathematics and Statistics\\[-0.8ex]
\small York University\\[-0.8ex] 
\small Toronto, ON, U.S.A.\\
\small\tt bergeron@yorku.ca, aram.dermenjian@gmail.com, machacek@yorku.ca\\}

\begin{document}

\maketitle


\begin{abstract}
For any $n > 0$ and $0 \leq m < n$, let $P_{n,m}$ be the poset of projective equivalence classes of $\{-,0,+\}$-vectors of length $n$ with sign variation bounded by $m$, ordered by reverse inclusion of the positions of zeros.
Let $\Delta_{n,m}$ be the order complex of $P_{n,m}$. 
A previous result from the third author shows that $\Delta_{n,m}$ is Cohen-Macaulay over $\mathbb{Q}$ whenever $m$ is even or $m = n-1$.
Hence, it follows that the $h$-vector of $\Delta_{n,m}$ consists of nonnegative entries.
Our main result states that $\Delta_{n,m}$ is partitionable and we give an interpretation of the $h$-vector when  $m$ is even or $m = n-1$.
When $m = n-1$ the entries of the $h$-vector turn out to be the new Eulerian numbers of type $D$ studied by Borowiec and M\l otkowski in [{\em Electron. J. Combin.}, 23(1):Paper 1.38, 13, 2016].
We then combine our main result with Klee's generalized Dehn-Sommerville relations to give a geometric proof of some facts about these Eulerian numbers of type $D$.
\end{abstract}

\section{Introduction}

%
%
%
In this paper we are interested in a special simplicial complex, $\Delta_{n,m}$ for  $n > 0$ and $0 \leq m < n$. This complex arose from the work of~\cite{signvarM} 
dedicated to a generalization of Postnikov's totally nonnegative Grassmannian~\cite{Pos}. Topologically the complex 
 $\Delta_{n,m}$ is a combinatorial manifold (with boundary)~\cite[Theorem 3.4]{signvarM} with geometric realization homotopy equivalent to $\mathbb{RP}^m$~\cite[Theorem 3.6]{signvarM}.
It follows that $\Delta_{n,m}$ is Cohen-Macaulay over $\mathbb{Q}$ if and only if $m$ is even or $m=n-1$~\cite[Corollary 3.7]{signvarM}.
As seen in~\cite{Sta}, when  a simplicial complex is Cohen-Macaulay its $h$-vector has nonnegative entries. This led us to investigate the combinatorial properties of $\Delta_{n,m}$.

Let us start with the simple example where $n=3$ and $m=2$. As depicted in Figure~\ref{fig:Delta32},
we can represent $\mathbb{RP}^2$ as the upper half sphere in $\mathbb{R}^3$ with the identification of the antipodal points along the equator. We take a cell decomposition of $\mathbb{RP}^2$ according
to the signs of the coordinates. Since we work on projective space, this is well defined up to a global sign, and we may choose the first nonzero coordinate to be positive. On $\mathbb{RP}^2$, we get 
the interior of four triangles (2-dimensional cells) that correspond to elements with the following sign vectors: $(+,+,+)$, $(+,-,+)$, $(+,+,-)$ and $(+,-,-)$. The six segments between those triangles (1-dimensional cells) correspond to the sign vectors: $(+,+,0)$, $(+,-,0)$, $(+,0,+)$, $(+,0,-)$, $(0,+,+)$ and $(0,+,-)$.
Finally, the three vertices (0-dimensional cells) are given by the sign vectors: $(+,0,0)$, $(0,+,0)$ and $(0,0,+)$. 
We then consider the poset $P_{3,2}$ of cells, ordered by $X \le Y$ if $X$ is in the closure of $Y$.
With the sign vectors, this corresponds to replacing some entries of the sign vector of $Y$ by zeros to obtain the sign vector of $X$.

The simplicial complex $\Delta_{3,2}$ is the order complex of the poset $P_{3,2}$. Geometrically that is the barycentric subdivision of the cells defining $P_{3,2}$ (see Figure~\ref{fig:Delta32}). 
If we look at the barycentric subdivision of the closure of $(+,+,+)$, then each face of the result can be assigned a permutation very naturally.
Notice that, given a face $X$, the coordinates $(x_1,x_2,x_3)$ of any point in $X$ will have the same relative ordering.
The permutation $\sigma$ assigned to $X$ is such the $\sigma(i)$ is the 
position of the $i$th smallest coordinate, reading equal coordinates from left to right. For example, $(1,2,1)$ has permutation $(1,3,2)$. In Figure~\ref{fig:Delta32} we give the permutation of the six facets and point
toward the smallest face with the same permutation. It turns out that all faces with the same permutation $\sigma$ correspond exactly to the interval of faces between the facet indexed by $\sigma$ and the (unique) minimal one.
The full complex $\Delta_{3,2}$ has 24 facets that are in bijection with the signed permutations of type $D_3$ (as a subgroup of signed permutations of type $B_3$). In this paper we will give a map
such that each face of $\Delta_{3,2}$ is assigned a type $D_3$ permutation inducing a decomposition of the face poset of $\Delta_{3,2}$ into Boolean intervals.

%
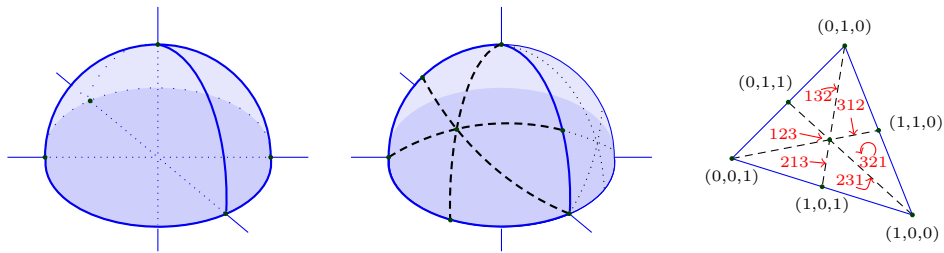
\begin{figure}[h!]
    \begin{center}
    \begin{tikzpicture}%
	[x={(1cm, .0cm)},
	y={(.0cm, 1 cm)},
	z={(.6cm, -0.5cm)},
	scale=0.05,
	font=\tiny,
	edge/.style={color=blue!95!black},
	facet/.style={fill=blue!95!black,fill opacity=0.100000},
	vertex/.style={inner sep=0pt,circle,draw=green!25!black,fill=green!75!black,thick,anchor=base}]
\draw[edge,dotted] (1.5,0,0) -- (30,0,0);
\draw[edge] (30,0,0) -- (40,0,0);
\draw[edge,dotted] (-1.5,0,0) -- (-30,0,0);
\draw[edge] (-30,0,0) -- (-40,0,0);
\draw[edge,dotted] (0,1,0) -- (0,30,0);
\draw[edge] (0,30,0) -- (0,40,0);
\draw[edge,dotted] (0,-1,0) -- (0,-18,0);
\draw[edge] (0,-19,0) -- (0,-25,0);
\draw[edge,dotted] (0,0,1) -- (0,0,30);
\draw[edge] (0,0,30) -- (0,0,40);
\draw[edge,dotted] (0,0,-1) -- (0,0,-30);
\draw[edge,loosely dotted] (0,0,-30) -- (0,0,-38);
\draw[edge] (0,0,-38) -- (0,0,-45);
 \draw[edge,facet,thick] (30,0,0) arc(5:-185:30 and 17);
 \draw[edge,facet] (30,0,0) arc(5:-185:30 and 17);
 \draw[edge,loosely dotted,facet] (-30,0,0) arc(185:-5:30 and 17);
 \draw[edge,thick] (0,30,0) arc(85:-5:20 and 41);
 \draw[edge,loosely dotted](0,30,0) arc(85:150:20 and 30);
 \draw[edge,facet,thick] (30,0,0) arc(0:180:30);
\node[vertex] at (30,0,0)     {\  };
\node[vertex] at (-30,0,0)     {\ };
\node[vertex] at (0,30,0)     {\ };
\node[vertex] at (0,0,30)     {\ };
\node[vertex] at (0,0,-30)     {\ };
\end{tikzpicture}
\quad
    \begin{tikzpicture}%
	[x={(1cm, .0cm)},
	y={(.0cm, 1 cm)},
	z={(.6cm, -0.5cm)},
	scale=0.05,
	font=\tiny,
	edge/.style={color=blue!95!black},
	facet/.style={fill=blue!95!black,fill opacity=0.100000},
	vertex/.style={inner sep=0pt,circle,draw=green!25!black,fill=green!75!black,thick,anchor=base}]
\draw[edge] (30,0,0) -- (40,0,0);
\draw[edge] (-30,0,0) -- (-40,0,0);
\draw[edge] (0,30,0) -- (0,40,0);
\draw[edge] (0,-19,0) -- (0,-25,0);
\draw[edge] (0,0,30) -- (0,0,40);
\draw[edge] (0,0,-38) -- (0,0,-45);
 \draw[edge,facet] (30,0,0) arc(5:-185:30 and 17);
 \draw[edge,facet] (30,0,0) arc(5:-185:30 and 17);
 \draw[edge,thick] (0,0,30) arc(-53:-185:30 and 17);
 \draw[edge,draw=none,facet] (-30,0,0) arc(185:-5:30 and 17);
 \draw[edge,thick] (0,30,0) arc(85:-5:20 and 41);
 \draw[edge,facet] (30,0,0) arc(0:180:30);
 \draw[edge,thick] (0,30,0) arc(90:180:30);
 \draw[thin, densely dashed,thick] (0,30,0) arc(95:185:15 and 43);
 \draw[thin, densely dashed,thick] (0,0,30) arc(-110:-150:75 and 84);
 \draw[thin, densely dashed,thick] (-30,0,0) arc(125:75:55 and 50);
 \draw[thin,  dotted] (1.9,19.1,24)   arc(75:60:55 and 50);
 \draw[thin, dotted] (0,30,0) arc(85:3:30 and 41);
 \draw[thin, dotted] (0,0,30) arc(-40:-5:38 and 50);
\node[vertex] at (-30,0,0)     {\ };
\node[vertex] at (0,30,0)     {\ };
\node[vertex] at (0,0,30)     {\ };
\node[vertex] at (-21,21.2,0)     {\ };
\node[vertex] at (-33.6,0,33.4)     {\ };
\node[vertex] at (1.9,19.1,24)     {\ };
\node[vertex] at (-19.7,13.9,13)     {\ };
\end{tikzpicture}
\quad
    \begin{tikzpicture}%
	[x={(1cm, .0cm)},
	y={(.0cm, 1 cm)},
	z={(.6cm, -0.5cm)},
	scale=0.05,
	font=\tiny,
	edge/.style={color=blue!95!black},
	vertex/.style={inner sep=0pt,circle,draw=green!25!black,fill=green!75!black,thick,anchor=base}]
\draw[edge] (-30,0,0)--(0,30,0);
\draw[edge] (-30,0,0)--(0,0,30);
\draw[edge] (0,30,0)--(0,0,30);
\draw[thin, densely dashed] (-30,0,0) -- (0,15,15);
\draw[thin, densely dashed] (0,30,0) -- (-15,0,15);
\draw[thin, densely dashed] (0,0,30) -- (-15,15,0);
\node[vertex] at (-30,0,0)     {\ };  \node at (-30,-5,0) {(0,0,1)};
\node[vertex] at (0,30,0)     {\ };  \node at (0,35,0) {(0,1,0)};
\node[vertex] at (0,0,30)     {\ }; \node at (0,-5,30) {(1,0,0)};
\node[vertex] at (0,15,15)     {\ };  \node at (10,17,15) {(1,1,0)};
\node[vertex] at (-15,0,15)     {\ };  \node at (-15,-5,15) {(1,0,1)};
\node[vertex] at (-15,15,0)     {\ }; \node at (-15,15,-10) {(0,1,1)};
\node[vertex] at (-10,10,10)     {\ };
\node[color=red!99!black] at (-18.3,8.3,3.3) {123};
\node[color=red!99!black] at (-9,18.3,3.3) {132};
\node[color=red!99!black] at (-18.3,3.3,8.3) {213};
\node[color=red!99!black] at (-9.5,3.5,18.3) {231};
\node[color=red!99!black] at (-3.3,8.3,18.3) {321};
\node[color=red!99!black] at (-3.3,18.3,8.3) {312};
\draw[color=red!99!black,->] (-14,8.3,3.3)--(-12,10.4,10) ;
\draw[color=red!99!black,->] (-3.3,15,8.3)--(-5,12.5,12.5) ;
\draw[color=red!99!black,->] (-14.3,3.3,8.3)--(-12.5,5,12.5) ;
\draw[color=red!99!black,->] (-8,1.5,18.3)..controls (-5,0.5,18.3) ..(-5,5,20) ;
\draw[color=red!99!black,->] (-8.3,20,3.3)..controls  (-5,20,2) .. (-5,21,5) ;
\draw[color=red!99!black,->] (-3.3,10,18.3) arc (-45:210:2) ;
\end{tikzpicture}
        \caption{On the left we have the cell decomposition of $\mathbb{RP}^2$ according the signs of coordinates whose closure poset is $P_{3,2}$. There are 4 interior triangles, 6 segments and 3 vertices. 
        In the center, we show the barycentric subdivision and obtain the simplicial complex $\Delta_{3,2}$ with 24 triangles, 36 segments and 13 vertices. On the right we look at the facet $(+,+,+)$ of $P_{3,2}$ and see that the facets of the barycentric subdivision are naturally indexed by permutations giving rise to a decomposition into Boolean intervals.}
        \label{fig:Delta32}
    \end{center}
\end{figure}

More generally, we show that the simplicial complexes $\Delta_{n,m}$ are partitionable when  $m$ is even or $m = n-1$.
This will give an interpretation for their $h$-vectors in terms of descents in even signed permutations.
For any $n > 0$ and $0 \leq m < n$, the simplicial complex $\Delta_{n,m}$
 is the order complex of a poset $P_{n,m}$.
The elements of $P_{n,m}$ are projective sign vectors of length $n$ with sign variation bounded by $m$.
Our main result is Theorem~\ref{thm:partitionable} which states that $\Delta_{n,m}$ is partitionable and gives an interpretation of the $h$-vector when  $m$ is even or $m = n-1$.
When $m = n-1$ the entries of the $h$-vector turn out to be the new Eulerian numbers of type $D$ studied by Borowiec and M\l otkowski~\cite{typeD}.
In Corollary~\ref{cor:DS} we combine our main result with Klee's generalized Dehn-Sommerville relations to give a geometric proof of some facts about these Eulerian numbers of type $D$.

There are at least two ways to view the motivation of this paper.
The first is that we want to further understand the complexes $\Delta_{n,m}$  by showing that they are partitionable and determining their $h$-vectors.
Given that $\Delta_{n,m}$ is Cohen-Macaulay if $m$ is even or $m = n-1$ it is natural to look for a partitioning since (even though it has been disproven~\cite{DGKM}) a long standing conjecture would suggest the complex may be partitionable~\cite[Conjecture 2.7]{Sta}.
It is not possible to show the stronger result that $\Delta_{n,m}$ is shellable for $m > 0$ since the complex is a manifold (with boundary) that is neither a ball nor a sphere~\cite[Proposition\ 1.2]{DanarajKlee}.
The second motivation is that our results give a geometric model for the new type-$D$ Eulerian numbers~\cite{typeD}.
It is well-known that the classical Eulerian numbers of type $A$ as well as Eulerian numbers of other types show up as the $h$-vector of the Coxeter complex (see e.g. Exercise 16 of Chapter 3 in~\cite{BB}).

\section{Sign variation and descents}

\subsection{Sign variation posets and complexes}
We will let $\V_n = \{-,0,+\}^n$ denote the set of \emph{sign vectors} of length $n$.
Given a sign vector $\omega$ the \emph{sign variation} of $\omega$ is denoted $\var(\omega)$ and is the number of times $\omega$ changes sign where zeros are ignored.
As an example we have that $\var((+,-,0,-,+)) = 2$.
The \emph{weight} of a sign vector $\omega$ is denoted $\wt(\omega)$ and is defined to be the number of non-zero entires of $\omega$.

For any $\omega \in \V_n$  we have $-\omega \in \V_n$.
We define an equivalence relation $\sim$ on $\V_n$ where $\omega \sim \omega'$  if and only if $\omega = \omega'$ or $\omega = - \omega'$.
We will let $\PV_n = (\V_n \setminus \{0\}^n) / \sim$ which is the collection of nonzero sign vectors up to equivalence.
Sign variation is well defined on $\PV_n$ since $\var(\omega) = \var(-\omega)$.

We will let $P_{n,m}$ denote the poset whose underlying set is $\{\omega \in \PV_n : \var(\omega) \leq m\}$ with order relation $\omega' < \omega$ if and only if $\pm \omega'$ can be obtained from $\omega$ by replacing some elements with $0$.
As examples $(0,+,0,-) < (+,+,+,-)$ and also $(0,+,0,-) < (+,-,-,+)$ since $(0,-,0,+) \sim (0,+,0,-)$.
The poset $P_{n,m}$ is ranked where the rank of an element $\omega$ is $\wt(\omega) - 1$.

A \emph{simplicial complex} is a collection of sets such that if $\sigma \in \Delta$ and $\tau \subseteq \sigma$ then $\tau \in \Delta$.
Notice this means that $\varnothing \in \Delta$ for any simplicial complex $\Delta$.
An element $\sigma \in \Delta$ is called a \emph{face} and the \emph{dimension} of $\sigma$ is $\dim \sigma = |\sigma| - 1$.
Faces which are maximal with respect to inclusion are called \emph{facets}.
The \emph{order complex} $\Delta(P)$ of a poset $P$ is the simplicial complex on vertex set $P$ whose $k$-dimensional faces are the chains of consisting of $k+1$ elements in $P$.
We then let $\Delta_{n,m}$ denote the order complex $\Delta(P_{n,m})$.

\begin{example}
    Let $n = 2$ and $m =1$.
    Then
    \[
        \V_n = \left\{ (+,+),\,(+,-),\,(+,0),\,(0,+),\,(0,-),\,(0,0),\,(-,+),\,(-,-),\,(-,0) \right\}.
    \]
    By our equivalence relation we have
    \[
        \PV_n = \left\{ (+,+),\,(+,-),\,(+,0),\,(0,+) \right\}.
    \]
    Since we can change $0$ to either a $+$ or a $-$, then the Hasse diagram of the poset $P_{2,1}$ is given on the left of Figure~\ref{fig:posSignVecs}.
	Looking at chains in our poset $P_{2,1}$ we see that we have four $1$ element chains and four $2$ element chains.
    Therefore, our order complex $\Delta_{2,1}$ has four $0$-dimensional faces and four $1$-dimensional faces.
    The Hasse diagram of the face poset of $\Delta_{2,1}$ is shown on the right of Figure~\ref{fig:posSignVecs}.
\end{example}

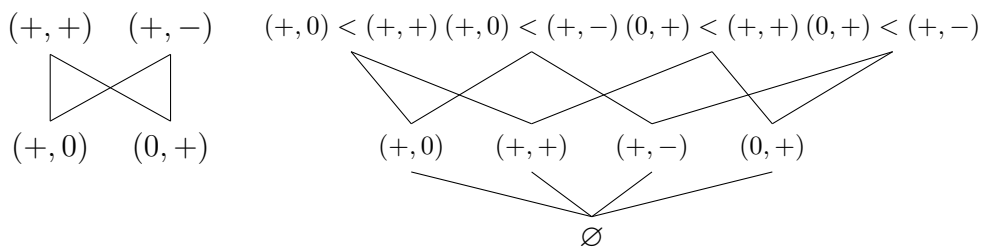
\begin{figure}[h!]
    \begin{center}
\begin{tikzpicture}[scale=0.8]
        \begin{scope}[shift={(0,0)}]        
            \node (p0) at (-1,0) {$(+,0)$};
            \node (0p) at (1,0) {$(0,+)$};
            \node (pp) at (-1,2) {$(+,+)$};
            \node (pn) at (1,2) {$(+,-)$};
            
            \draw (p0.north) -- (pp.south);
            \draw (p0.north) -- (pn.south);
            \draw (0p.north) -- (pp.south);
            \draw (0p.north) -- (pn.south);
        \end{scope}
        \begin{scope}[shift={(8,0)}]
                    \node (min) at (0,-1.5) {$\varnothing$};

            \node (p0) at (-3,0) {\footnotesize$(+,0)$};
            \node (0p) at (3,0) {\footnotesize$(0,+)$};
            \node (pp) at (-1,0) {\footnotesize$(+,+)$};
            \node (pn) at (1,0) {\footnotesize$(+,-)$};
            
            \node (p0-pp) at (-4,2) {\footnotesize$(+,0)<(+,+)$};
            \node (p0-pn) at (-1,2) {\footnotesize$(+,0)<(+,-)$};
            \node (0p-pp) at (2,2) {\footnotesize$(0,+)<(+,+)$};
            \node (0p-pn) at (5,2) {\footnotesize$(0,+)<(+,-)$};

	\draw (min.north) -- (p0.south);
		\draw (min.north) -- (0p.south);
	\draw (min.north) -- (pp.south);
	\draw (min.north) -- (pn.south);
            \draw (p0.north) -- (p0-pp.south);
            \draw (p0.north) -- (p0-pn.south);
            \draw (0p.north) -- (0p-pp.south);
            \draw (0p.north) -- (0p-pn.south);
            \draw (pp.north) -- (p0-pp.south);
            \draw (pp.north) -- (0p-pp.south);
            \draw (pn.north) -- (p0-pn.south);
            \draw (pn.north) -- (0p-pn.south);
        \end{scope}
        
    \end{tikzpicture}
        \caption{Hasse diagram of $P_{2,1}$ and the face poset $\mathcal{F}(\Delta_{2,1})$ of the order complex.}
        \label{fig:posSignVecs}
    \end{center}
\end{figure}

For a simplicial complex of dimension $d$, let $f_i$ denote the number of $i$-dimensional faces.
The \emph{$f$-vector} of a simplicial complex $\Delta$ is then the $f_i$ arranged as a vector:
\[
    f(\Delta) = (f_{-1}, f_0, \ldots, f_d)
\]
where $f_{-1} = 1$.
The $h$-vector of a simplicial complex $\Delta$ is defined using the $f$-vector.
Let
\[
    h_k = \sum_{i = 0}^k (-1)^{k-i} \binom{d-i}{k-i} f_{i-1}.
\]
Then the \emph{$h$-vector} of $\Delta$ is the vector:
\[
    h(\Delta) = \left( h_0, h_1, \ldots, h_{d+1} \right).
\]
Looking at Example~~\ref{fig:posSignVecs} we see that $f(\Delta_{2,1}) = (1,4,4)$ and $h(\Delta_{2,1}) = (1,2,1)$.

The \emph{face poset} of a simplicial complex $\Delta$ is denoted $\mathcal{F}(\Delta)$ and consists of all faces of $\Delta$ ordered by inclusion.
Given a poset $P$ and any two elements $x,y \in P$ with $x \leq y$ we have the \emph{(closed) interval} 
\[[x,y] = \{z \in P : x \leq z\leq y\}.\]
The collection of all subsets of a given set ordered by inclusion is known as a \emph{Boolean poset}.
A simplicial complex $\Delta$ is said to be \emph{partitionable} if its face poset can be written as the disjoint union
\[\mathcal{F}(\Delta) = \bigsqcup_{F\in Facets(\Delta)} [G_F, F] \]
where  $Facets(\Delta)$ is the set of facets (maximal faces) of $\Delta$ and each interval $[G_F, F]$ is a Boolean poset for some $G_F$.
In general the $h$-vector of a simplicial complex may contain negative entries.
However, if $\Delta$ is partitionable with its face poset written as above, then by a result of Stanley (see \cite{Sta}):
\[h_j = | \{F : |G_F| = j \text { and } F\in Facets(\Delta)\} |\]

Each $\Delta_{n,m}$ is a combinatorial manifold (with boundary)~\cite[Theorem 3.4]{signvarM} with geometric realization homotopy equivalent to $\mathbb{RP}^m$~\cite[Theorem 3.6]{signvarM}.
The geometric realization of $\Delta_{n,n-1}$ is the manifold $\mathbb{RP}^{n-1}$.
It follows that $\Delta_{n,m}$ is Cohen-Macaulay over $\mathbb{Q}$ if and only if $m$ is even or $m=n-1$~\cite[Corollary 3.7]{signvarM}.
When a simplicial complex is Cohen-Macaulay its $h$-vector has nonnegative entries.
For a treatment of Cohen-Macaulay simplicial complexes and their properties we refer the reader to~\cite{Sta}.


\subsection{Signed permutations and descents}
We will denote the set of permutations of $[n]$ by $S^A_n$ and usually think of permutations in one-line notation.
A \emph{signed permutation of $[n]$} is a bijection $\pi: [\pm n] \to [\pm n]$ such that $\pi(-i) = - \pi(i)$ for all $i \in [\pm n]$.
The set of signed permutations of $[n]$ is denoted by $S^B_n$.
Any signed permutation $\pi$ can be represented by the sequence $[\pi(1), \pi(2), \dots, \pi(n)]$ which is known as \emph{window} notation.
We will often use $\bar{i}$ to denote $-i$ for $i \in [n]$.
In this way we can denote the signed permutation with window notation $[-2,3,4,-1]$ by $\bar{2}34\bar{1}$.
We will also denote a signed permutation by an ordered pairs $(\pi, X)$ consisting of a usual permutation $\pi \in S^A_n$ along with a set $X \subseteq [n]$ recording the \emph{numbers} of negative entries in window notation, thus we can denote $[-2, 3, 4, -1]$ by $\left( 2341, \left\{ 1,2 \right\}\right)$.

We let $S^D_n$ denote the set of \emph{even signed permutations} of $[n]$ which means there is an even number of negative entries in the window notation.
Equivalently we can say 
\[S^D_n = \{(\pi, X) : \pi \in S^A_n, X \subseteq [n], |X| \equiv 0 \bmod 2\}.\]
We then let $S^D_{n,m}$ denote all elements of $S^D_n$ with at most $m$ negative entries in window notation or equivalently
\[S^D_{n,m} = \{(\pi, X) : \pi \in S^D_n, |X| \leq m\}.\]

For any sequence of integers $w = (w_0, w_1, \dots, w_n)$  we say that $i$ is a \emph{descent} of $w$ if $w_i > w_{i+1}$.
For any signed permutation $\pi$ with window notation $[\pi(1), \pi(2), \dots, \pi(n)]$ we let $w(\pi) = (0, \pi(1), \pi(2), \dots, \pi(n))$ and define
\[\Des(\pi) = \{i : i \text{ is a descent of } w(\pi)\} \subseteq \{0,1,\ldots n-1\}\]
to be the \emph{descent set} of $\pi$.
We also let $\des(\pi) = |\Des(\pi)|$.
Finally we let \[D(n,k) = |\{\pi \in S^D_n: \des(\pi) = k\}|\]
which count the number of even signed permutations with a given number of descents.

\begin{remark}
The quantity $D(n,k)$ was first studied by Borowiec and M\l otkowski~\cite{typeD}.
There is a general notion of descent in any Coxeter group.
The number $D(n,k)$ computes descents with respect to the Coxeter group generators of $S^B_n$ restricted to elements the subgroup $S^D_n$.
For the general theory of descents as well as other combinatorics in Coxeter groups we refer the reader to~\cite{BB}.
\end{remark}

\section{The partitioning}
Given $\omega \in \PV_n$ we will consider indices cyclically so that $\omega_i = \omega_{i+n}$ for any $i$.
We say that $i \in [n]$ is a \emph{cyclic sign flip} of $\omega \in \PV_n$ if there exists a $j$ such that $\omega_{i-j} \omega_i < 0$ while $\omega_{i-k} \omega_i = 0$ for all $1 \leq k < j$.
We define a function $\BAR: \PV_n \to [n]$ by
\[\BAR(\omega) = \{i \in [n] : i \text{ is a cyclic sign flip of $\omega$}\}\]
for each $\omega \in \PV_n$.
For example, $\BAR((0,+,-,-,0,+,-)) = \{2,3,6,7\}$.  Here $2 \in \BAR((0,+,-,-,0,+,-))$ since $\omega_2 =+$ and for $j=2$, as we are looking at the indices cyclically, $\omega_0=\omega_7 = -$.
We have the following results which follow immediately from the definition of $\BAR$.

\begin{lemma}
For any $\omega \in \PV_n$ the size of $\BAR(\omega)$ is even.
\label{lem:even}
\end{lemma}

\begin{lemma}
If $\omega \in \PV_n$  and $i \in \BAR(\omega)$, then $\omega_i \neq 0$.
\label{lem:nonzero}
\end{lemma}

\begin{definition}
We define a function $\Phi: \Delta_{n,m} \to S^D_n$ for any $ 0 \leq m \leq n-1$.
Consider any chain $C: \omega^{(1)} < \omega^{(2)} < \cdots < \omega^{(r)}$.
To obtain $\Phi(C)$ we start with the empty word, setting $\omega^{(0)} = 0^n$.
For $s = 1,2, \ldots , r$ iterate the following process:
\begin{enumerate}
\item Set $I_s := \{i \in [n]: \omega^{(s)}_i \neq 0, \omega^{(s-1)}_i = 0\}$.
\item Set $\bar{I_s} := \{i : i \in I_s, i \not\in \BAR(\omega^{(r)})\} \cup \{\bar{i} : i \in I_s, i \in \BAR(\omega^{(r)})\}$.
\item Let $\omega'_s$ be the word where the elements of $\bar{I_s}$ are written in increasing order.
\end{enumerate}
Finally set $I_{r+1} := \{i \in [n]: \omega^{(r)}_i = 0\}$.
Let $\bar{I_{r+1}}$ and $\omega'_{r+1}$ be defined as above.
Then we obtain $\Phi(C)$ by concatenating all the words in reverse order:
\[
    \Phi(C) = \omega'_{r+1}\omega'_{r} \ldots \omega'_2 \omega'_1.
\]

For $\pi = \Phi(C) = \omega'_{r+1}\omega'_{r} \ldots \omega'_2 \omega'_1$ and $0 \leq i \leq r+1$ we let $\ell(C,i)$ denote the number of letters in the initial part of $\pi$, i.e.~$\ell(C,i) = \abs{\cup_{j > i} I_j}$.

\label{def:Phi}
\end{definition}

\begin{example}
    Let $n = 9$ and $m = 8$.
    As a first example, take 
    \[C_1: (0,+,-,0,0,0,0,0,+) < (0,+,-,0,-,+,0,0,+) < (0,+,-,-,-,+,-,+,+).\]
Then
\[\BAR(\omega^{(3)}) = \BAR((0,+,-,-,-,+,-,+,+)) = \{3,6,7,8\}.\]
For $s = 1$ we have:
\begin{enumerate}
    \item $I_1 = \left\{ 2,3,9 \right\}$
    \item $\bar{I}_1 = \left\{ 2,\bar{3},9 \right\}$
    \item Therefore, $\omega'_1 = \bar{3}29$.
\end{enumerate}
For $s = 2$ we have:
\begin{enumerate}
    \item $I_2 = \left\{ 5,6 \right\}$
    \item $\bar{I}_2 = \left\{ 5,\bar{6} \right\}$
    \item Therefore, $\omega'_2 = \bar{6}5$.
\end{enumerate}
For $s = 3$ we have:
\begin{enumerate}
    \item $I_3 = \left\{ 4,7,8 \right\}$
    \item $\bar{I}_3 = \left\{ 4, \bar{7}, \bar{8} \right\}$
    \item Therefore, $\omega'_3 = \bar{8}\bar{7}4$.
\end{enumerate}
Finally $I_4 = \left\{ 1 \right\}$.
Therefore $\omega'_4 = 1$.

Concatenating these results gives:
\[
    \Phi(C_1) = \omega'_4\omega'_3\omega'_2\omega'_1 = 1\bar{8}\bar{7}4\bar{6}5\bar{3}29.
\]

Then we have:
\[
    \ell(C_1,4) = 0,\,\ell(C_1,3) = 1,\,\ell(C_1,2) = 4,\,\ell(C_1,1) = 6,\,\ell(C_1,0) = 9 = n.
\]

Similarly, if we take
\[C_2: (0,+,-,0,0,0,0,0,-) < (0,+,-,0,-,+,0,0,-) < (0,+,-,-,-,+,-,+,-)\]
we have 
\[\BAR((0,+,-,-,-,+,-,+,-)) = \{2,3,6,7,8,9\}\]
and $\Phi(C_2) = 1\bar{8}\bar{7}4\bar{6}5\bar{9}\bar{3}\bar{2}$.
\end{example}

\begin{lemma}
If $0 \leq m \leq n-1$ such that $m$ is even, then $\Phi(\Delta_{n,m}) \subseteq S^D_{n,m}$.
\label{lem:subset}
\end{lemma}
\begin{proof}
Let $C: \omega^{(1)} < \omega^{(2)} < \cdots < \omega^{(r)}$ be any chain in $\Delta_{n,m}$.
Also let $k = \var(\omega^{(r)})$.
Now $|\BAR(\omega^{(r)})| = k$ if $k$ is even and $|\BAR(\omega^{(r)})| = k+1$ if $k$ is odd.
Since $k \leq m$ where $m$ is even it follows that $|\BAR(\omega^{(r)})| \leq m$.
Therefore, $\Phi(\Delta_{n,m}) \subseteq S^D_{n,m}$.
\end{proof}

\begin{lemma}
For any $n$
\[ \{\text{Facets of } \Delta_{n,n-1}\} \overset{\Phi}\longrightarrow S^D_{n}\]
is a bijection and thus $\Phi(\Delta_{n,n-1}) = S^D_{n}$.
Moreover, for $0 \leq m \leq n-1$ such that $m$ is even, then
\[ \{\text{Facets of } \Delta_{n,m}\} \overset{\Phi}\longrightarrow S^D_{n,m}\]
is a bijection and $\Phi(\Delta_{n,m}) = S^D_{n,m}$.
\label{lem:bijection}
\end{lemma}
\begin{proof}
Consider any $n$ and $0 \leq m \leq n-1$ with $m$ even or $m = n-1$.
Given any chain $C: \omega^{(1)} < \omega^{(2)} < \cdots < \omega^{(r)}$  in $\Delta_{n,m}$ we see that $\Phi(C) \subseteq S^D_{n,m}$ by Lemma~\ref{lem:subset}.
So, showing the bijection on the facets will imply that $\Phi(\Delta_{n,m}) = S^D_{n,m}$.

Let us describe the bijection between facets of $\Delta_{n,n-1}$ and elements of $S^D_n$.
This bijection will restrict to a bijection between facets of $\Delta_{n,m}$ and elements of $S^D_{n,m}$ whenever $m$ is even.
Any facet of $\Delta_{n,n-1}$ is a saturated chain
\[C:\omega^{(1)} < \omega^{(2)} < \cdots < \omega^{(n)}\]
with $\omega^{(i)} \in \PV_n$ for $1 \leq i \leq n$.
Such a statured chain $C$ determines a permutation $\pi_C$ whose $i$th entry for $1 \leq i < n$ in one-line notation is the unique index $k$ such that $\omega^{(n-i+1)}_k \neq 0$ but $\omega^{(n-i)}_k = 0$  while its $n$th entry is the unique index $k$ such that $\omega^{(1)}_k \neq 0$.
Also, the statured chain $C$ determines the set $X_C = \BAR(\omega^{(n)})$.
So, the desired bijection maps $C$ to $(\pi_C, X_C)$ which is indeed an element of $S^D_n$ by Lemma~\ref{lem:even}.
Moreover, it is clear that $(\pi_C, X_C)$ is in $S^D_{n,m}$ if $C$ is a chain in $\Delta_{n,m}$.

To see this map is a bijection we describe the inverse map.
For any $\pi \in S^D_n$ consider $\pi$ as an element of $S^A_n$ by forgetting the sign of entries in window notation.
This determines a saturated chain of binary vectors in the usual way where permutations correspond to saturated chains in a Boolean poset while subsets can be put into bijective correspondence with binary vectors.
The binary vectors can be made into elements of $\PV_n$ in the only way compatible with negative entries of $\pi$ in window notation.
This is done by taking the set $X$ of negative entries of $\pi$ and creating the unique sign vector in $\PV_n$ with no $0$ components, whose set of cyclic sign flips is equal to $X$.
An example of this bijection can be seen in Example~\ref{ex:bijection}.
\end{proof}

\begin{example}
The bijection in the proof of Lemma~\ref{lem:bijection} maps $\bar{2}315\bar{4}$, which is equivalent to the pair $(23154, \{2,4\})$, to the chain
\[(0,0,0,+,0) < (0,0,0,+,+) < (+,0,0,+,+) < (+,0,-,+,+) < (+,-,-,+,+)\]
that is a facet of $\Delta_{5,2}$. In more details, it maps the permutation $23154$ to the Boolean chain 
\[(0,0,0,1,0) < (0,0,0,1,1) < (1,0,0,1,1) < (1,0,1,1,1) < (1,1,1,1,1).\]
Then the set $ \{2,4\}$ determines (uniquely, up to a global sign) the sign changes as $(+,-,-,+,+)$.
\label{ex:bijection}
\end{example}

For a chain $C: \omega^{(1)} < \omega^{(2)} < \cdots < \omega^{(k)}$, let $C_i$ denote the subchain of $C$ with $\omega^{(i)}$ removed:
\[
    C_i : \omega^{(1)} < \cdots < \omega^{(i-1)} < \omega^{(i+1)} < \cdots < \omega^{(k)}.
\]
For $I \subseteq [k]$ let $C_I$ be the subchain of $C$ with $\omega^{(i)}$ removed for  all $i \in I$.
The sign vectors we can remove from a chain $C$, without changing the value of $\Phi(C)$, are directly governed by the descent set of the permutation to which it is associated.
The reader is invited to recall Definition~\ref{def:Phi} for the notation used in the following lemma and proof.

\begin{lemma}
    Let $0 \leq m \leq n-1$ such that either $m$ is even or $m = n-1$.
    Let 
    \[
        C: \omega^{(1)} < \omega^{(2)} < \cdots < \omega^{(k)}
    \]
    be a chain in $\Delta_{n,m}$ and let $\Phi(C) = \pi = \omega_{k+1}'\omega_k'\omega_{k-1}'\ldots\omega_2'\omega_1' \in \Phi(\Delta_{n,m})$ with descent set $\Des(\pi)$.
    
    For $i \in [k-1]$, $\Phi(C_i) = \pi = \Phi(C)$ if and only if $\ell(C,i) \notin \Des(\pi)$.
    For $i =k$, $\Phi(C_k) = \pi = \Phi(C)$ if and only if $\ell(C,k) \notin \Des(\pi)$ and $\BAR(\omega^{(k)}) = \BAR(\omega^{(k-1)})$.
    \label{lem:descent_or_remove}
\end{lemma}
\begin{proof}
    Suppose first that $i < k$.
    Recall that $C_i$ is the subchain of $C$ with $\omega^{(i)}$ removed.
    Let $I_s$, $\bar{I_s}$ and $\omega'_s$ be the maps used in Definition~\ref{def:Phi} for $C$ and let $J_s$, $\bar{J_s}$ and ${v}'_{s}$ be the corresponding maps for $C_i$.
    By construction of $\Phi$, it is clear that $I_s = J_s$, $\bar{I_s} = \bar{J_s}$ and $\omega'_s = {v}'_{s}$ for all $s < i$ and $I_s = J_{s-1}$, $\bar{I_s} = \bar{{J}_{s-1}}$ for all $s > i+1$ and $\omega'_s = {v}'_{s-1}$ for $s > i+1$.
    It suffices to show that $\omega'_{i+1} \omega'_{i} = {v}'_{i}$ if and only if $\ell(C,i) \notin \Des(\pi)$.

    If $\omega'_{i+1}  \omega'_{i}  = {v}'_{i}$ then for every $j \in \bar{I_i}$ and $k \in \bar{I_{i+1}}$ we have $j > k$ and therefore there is no descent at $\ell(C,i)$.
    Similarly if there is no descent at $\ell(C,i)$, then we can add an arbitrary cut in ${v}_i$ and split it into $\omega'_{i+1}\omega'_i$, giving us the desired result.

    Suppose next that $i = k$.
    This is similar to the previous case with the exception that if $\BAR(\omega^{(k)}) \neq \BAR(\omega^{(k-1)})$, then we no longer have $\bar{I_s} = \bar{{J}_s}$ for $s < k$ hence the additional requirement in the only if.
\end{proof}

By repeated applications of the previous lemma we have the following.
\begin{proposition}
    Let $0 \leq m \leq n-1$ such that either $m$ is even or $m = n-1$.
    For $C \in \Delta_{n,m}$ a chain with $k$ elements and $I \subseteq [k]$, then $\Phi(C_I) = \Phi(C)$ if and only if $I \cap \Des(\Phi(C)) = \emptyset$ and $\BAR(\omega_I) = \BAR(\omega)$ where $\omega$ and $\omega_I$ are the top sign vectors in $C_I$ and $C$ respectively.
    \label{prop:PhiEqual}
\end{proposition}

For any $\pi \in S_D^n$ we let $C^{\pi}$ denote the saturated chain which is in bijection with $\pi$ from Lemma~\ref{lem:bijection}.
Let $C_{\pi}$ denote the rank selected subchain of $C^{\pi}$ restricted to the ranks $\{n - i : i \in \Des(\pi)\}$.

\begin{example}
Considering $\pi = \bar{2}315\bar{4}$ we have
\[C^{\pi} : (0,0,0,+,0) < (0,0,0,+,+) < (+,0,0,+,+) < (+,0,-,+,+) < (+,-,-,+,+)\]
and
\[C_{\pi} :  (0,0,0,+,0) < (+,0,0,+,+) < (+,-,-,+,+)\]
since $\Des(\pi) = \{0,2,4\}$.
\label{ex:bijection2}
\end{example}

\begin{lemma}
If $\pi \in \Phi(\Delta_{n,m})$ for $0 \leq m \leq n-1$ such that $m$ is even or $m = n-1$, then $\Phi^{-1}(\pi) = [C_{\pi}, C^{\pi}]$ is a Boolean interval in $\mathcal{F}(\Delta_{n,m})$
\label{lem:BooleanInt}
\end{lemma}
\begin{proof}
    Suppose $\pi$ is an element in $S_n^D$ and let $\Des(\pi) = \left\{ d_1, d_2, \ldots d_k \right\}$ be the set of descents (in numerical order).
    We write $\pi$ in the following manner:
    \[
        \pi = \pi_{1,1}\pi_{1,2} \ldots \pi_{1,d_1} \pi_{2, 1} \pi_{2, 2} \ldots \pi_{2,d_2-d_1} \ldots \pi_{k,1} \pi_{k,2} \ldots \pi_{k,d_k-d_{k-1}}
    \]

    Let $C^\pi$ be the saturated chain which is in bijection with $\pi$ from Lemma~\ref{lem:bijection} and notice that we can label each element in the saturated chain $C$ by one of the $\pi_{i,j}$ as follows:
    \begin{align*}
        C^\pi :~&\omega^{(k,d_k - d_{k-1})} < \cdots <\omega^{(k,2)} <\omega^{(k,1)}\\
        &< \cdots \\
        &< \omega^{(2,d_2-d_1)} < \cdots < \omega^{(2,2)} < \omega^{(2,1)} \\
        &< \omega^{(1,d_1)} < \cdots <\omega^{(1,2)} < \omega^{(1,1)}
    \end{align*}

    By Lemma~\ref{lem:descent_or_remove}, we can inductively remove any $\omega^{(i,j)}$ as long as $j$ is not maximal in $\pi_i$ (or if $i = 1$, then we also require that the sign vector directly before has the same set of cyclic sign flips).
    Since the order of the removals doesn't alter our permutation, the ordering of these subchains is isomorphic to the subsets of a set ordered by inclusion, in other words the Boolean interval, with the bottom element being $C_\pi$, the rank selected subchain of $C^\pi$ restricted to the ranks $\{n - i : i \in \Des(\pi)\}$.
\end{proof}

\begin{lemma}
If $0 \leq m \leq n-1$ such that either $m$ is even or $m = n-1$, then $\Phi^{-1}(\pi_1) \cap \Phi^{-1}(\pi_2) = \emptyset$ for any $\pi_1, \pi_2 \in \Phi(\Delta_{n,m})$ with $\pi_1 \neq \pi_2$.
\label{lem:disjoint}
\end{lemma}
\begin{proof}
    This comes directly from Lemma~\ref{lem:BooleanInt} together with the characterization of removing sign vectors from chains in Lemma~\ref{lem:descent_or_remove}.
\end{proof}

\begin{figure}[h!]
    \begin{center}
    \begin{tikzpicture}
                    \node[color=red] (min) at (0,-1.5) {$\varnothing$};

            \node[color=red] (p0) at (-3,0) {$(0,+)$};
            \node[color=blue] (0p) at (3,0) {$(+,0)$};
            \node[color=red] (pp) at (-1,0) {$(+,+)$};
            \node[color=darkgray] (pn) at (1,0) {$(+,-)$};
            
            \node[color=red] (p0-pp) at (-4,2) {$\overset{12}{(+,0)<(+,+)}$};
            \node[color=darkgray] (p0-pn) at (-1,2) {$\overset{\bar{2}\bar{1}}{(+,0)<(+,-)}$};
            \node[color=blue] (0p-pp) at (2,2) {$\overset{21}{(+,0)<(+,+)}$};
            \node[color=purple] (0p-pn) at (5,2) {$\overset{\bar{1}\bar{2}}{(0,+)<(+,-)}$};

	\draw[color=red, ultra thick] (min.north) -- (p0.south);
		\draw[color=red, ultra thick] (min.north) -- (pp.south);
		\draw[dashed] (min.north) -- (0p.south);
		\draw[dashed] (min.north) -- (pn.south);
            \draw[color=red, thick] (p0.north) -- (p0-pp.south);
            \draw[dashed] (p0.north) -- (p0-pn.south);
            \draw[color=blue, ultra thick] (0p.north) -- (0p-pp.south);
            \draw[color=red, ultra thick] (pp.north) -- (p0-pp.south);
            \draw[dashed] (pp.north) -- (0p-pp.south);
            \draw[dashed] (pn.north) -- (0p-pn.south);
            \draw[dashed] (0p.north) -- (0p-pn.south);

            \draw[color=darkgray, ultra thick] (pn.north) -- (p0-pn.south);
    \end{tikzpicture}

        \caption{An example of the partitioning on $\mathcal{F}(\Delta_{2,1})$.}
        \label{fig:partition}
    \end{center}
\end{figure}
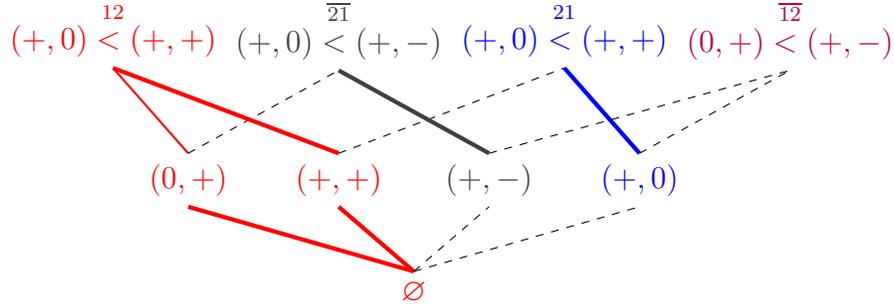

\begin{theorem}
If $0 \leq m \leq n-1$ such that either $m$ is even or $m = n-1$, then $\Delta_{n,m}$ is partitionable with
\[\mathcal{F}(\Delta_{n,m}) = \bigsqcup_{\pi \in \Phi(\Delta_{n,m})} [C_{\pi}, C^{\pi}]\]
and thus
\[h_j(\Delta_{n,m}) = | \{\pi \in \Phi(\Delta_{n,m}) : \des(\pi) = j \} |\]
for each $0 \leq j \leq n$.
\label{thm:partitionable}
\end{theorem}
\begin{proof}
By Lemma~\ref{lem:BooleanInt} and Lemma~\ref{lem:disjoint} it follows that $\Delta_{n,m}$ is partitionable whenever $m$ is even or $m = n-1$.
The $h$-vector equality follows from the partitionability.
\end{proof}

\begin{remark}
One interesting question that we leave open is the study of the {\sl flag}-$h$ vector of $\Delta_{m,n}$. 
Since $P_{n,m}$ is a graded poset we have a flag-$h$ vector and the content of Theorem~\ref{thm:partitionable} can be modified to describe the flag-$h$ vector.
Given a flag-$h$ vector, there is a natural quasisymmetric function assigned to it (see~\cite{ABS,Ehr96}).
It would be interesting to study this function, but our initial computation shows that it is not symmetric.
One may need to use different notions of quasisymmetric as in~\cite{QsymB}, but we leave this question to the interested reader.
\end{remark}

In Figure~\ref{fig:partition} we give an example of the partitioning given in Theorem~\ref{thm:partitionable}.
Above each facet we write the signed permutation given in the bijection from Lemma~\ref{lem:bijection}.

Next we give an application of Theorem~\ref{thm:partitionable} that uses Klee's generalization of the Dehn-Sommerville relations~\cite{Klee} which states that
\begin{equation}
h_{d-j} - h_j = (-1)^j \binom{d}{j} ((-1)^{d-1} \tilde{\chi}(\Delta) - 1)
\label{eq:DS}
\end{equation}
where $h = (h_0, h_1 ,\dots, h_d)$ is the $h$-vector of a $(d-1)$-dimensional simplicial complex $\Delta$ which is a (homology) manifold.

\begin{corollary}
If $n$ is even, then $D(n,j) = D(n,n-j)$ for all $0 \leq j \leq n$. If $n$ is odd, then $D(n,j) = D(n,n-j) + (-1)^j \binom{n}{j}$ for all $0 \leq j \leq n$.
\label{cor:DS}
\end{corollary}
\begin{proof}
By Theorem~\ref{thm:partitionable} we see that $h_j(\Delta_{n,n-1}) = D(n,j)$.
The geometric realization of $\Delta_{n,n-1}$ is the manifold $\mathbb{RP}^{n-1}$. 
It is known that
\[\tilde{\chi}(\Delta_{n,n-1}) = \begin{cases} -1 & n \equiv 0 \bmod 2; \\ 0 & n \equiv 1 \bmod 2. \end{cases} \]
Thus by applying~(\ref{eq:DS}) we find that for $n$ even
\[D(n,j) = h_j(\Delta_{n,n-1}) = h_{n-j}(\Delta_{n,n-1}) = D(n,n-j)\]
and for $n$ odd
\[D(n,j) = h_j(\Delta_{n,n-1})  = h_{n-j}(\Delta_{n,n-1}) + (-1)^j \binom{d}{j} = D(n,n-j) +  (-1)^j \binom{d}{j}.\]
for each $0 \leq j \leq n$.
\end{proof}

\begin{remark}
The content of Corollary~\ref{cor:DS} was previously known as it follows from~\cite[Proposition 4.1]{typeD} and~\cite[Proposition 4.3]{typeD}.
\end{remark}


\section{Acknowledgements}
The authors would like to thank Robin Sulzgruber for many valuable conversations.
The authors benefited from the working environment of the Algebraic Combinatorics Seminar at the Fields Institute.



\begin{thebibliography}{10}

\bibitem{ABS}
Marcelo Aguiar, Nantel Bergeron, and Frank Sottile.
\newblock Combinatorial {H}opf algebras and generalized {D}ehn-{S}ommerville
  relations.
\newblock {\em Compos. Math.}, 142(1):1--30, 2006.

\bibitem{BB}
Anders Bj\"{o}rner and Francesco Brenti.
\newblock {\em Combinatorics of {C}oxeter groups}, volume 231 of {\em Graduate
  Texts in Mathematics}.
\newblock Springer, New York, 2005.

\bibitem{typeD}
Anna Borowiec and Wojciech M{\l}otkowski.
\newblock New {E}ulerian numbers of type {$D$}.
\newblock {\em Electron. J. Combin.}, 23(1):Paper 1.38, 13, 2016.

\bibitem{DanarajKlee}
Gopal Danaraj and Victor Klee.
\newblock Shellings of spheres and polytopes.
\newblock {\em Duke Math. J.}, 41:443--451, 1974.

\bibitem{DGKM}
Art~M. Duval, Bennet Goeckner, Caroline~J. Klivans, and Jeremy~L. Martin.
\newblock A non-partitionable {C}ohen-{M}acaulay simplicial complex.
\newblock {\em Adv. Math.}, 299:381--395, 2016.

\bibitem{Ehr96}
Richard Ehrenborg.
\newblock On posets and {H}opf algebras.
\newblock {\em Adv. Math.}, 119(1):1--25, 1996.

\bibitem{QsymB}
Samuel~K. Hsiao and T.~Kyle Petersen.
\newblock Colored posets and colored quasisymmetric functions.
\newblock {\em Ann. Comb.}, 14(2):251--289, 2010.

\bibitem{Klee}
Victor Klee.
\newblock A combinatorial analogue of {P}oincar\'{e}'s duality theorem.
\newblock {\em Canadian J. Math.}, 16:517--531, 1964.

\bibitem{signvarM}
John Machacek.
\newblock Boundary measurement and sign variation in real projective space.
\newblock {\em Ann. Inst. Henri Poincar\'e D} (to appear).
\newblock arXiv:1909.04640 [math.CO].

\bibitem{Pos}
Alexander Postnikov.
\newblock Total positivity, {G}rassmannians, and networks.
\newblock arXiv:math/0609764 [math.CO].

\bibitem{Sta}
Richard~P. Stanley.
\newblock {\em Combinatorics and commutative algebra}, volume~41 of {\em
  Progress in Mathematics}.
\newblock Birkh\"{a}user Boston, Inc., Boston, MA, second edition, 1996.

\end{thebibliography}
\end{document}